\documentclass[a4paper]{amsart}
\usepackage{microtype}

\usepackage{amsmath,amstext,amssymb,mathrsfs,amscd,amsthm,indentfirst, bbm}
\usepackage{amsfonts}

\usepackage{booktabs}
\usepackage{verbatim}
\usepackage{enumerate}
\usepackage{graphicx}
\usepackage{textcomp}

\usepackage{tikz, tikz-cd}
\usetikzlibrary{matrix,calc,arrows}

\newcommand{\bQ}{\mathbb{Q}}
\newcommand{\bR}{\mathbb{R}}

\newcommand{\cL}{\mathcal{L}}

\newcommand\lra{\longrightarrow}

\newcommand\trf{\mathrm{trf}}
\newcommand\Diff{\mathrm{Diff}}

\renewcommand{\epsilon}{\varepsilon}

\newcommand{\SO}{\mathrm{SO}}

\newcommand{\GL}{\mathrm{GL}}

\mathchardef\ordinarycolon\mathcode`\:
\mathcode`\:=\string"8000
\begingroup \catcode`\:=\active
  \gdef:{\mathrel{\mathop\ordinarycolon}}
\endgroup

\allowdisplaybreaks

\usepackage{amsthm}
\theoremstyle{plain}

\newtheorem{theorem}{Theorem}[section]

\newtheorem{proposition}[theorem]{Proposition}
\newtheorem{lemma}[theorem]{Lemma}

\newtheorem{corollary}[theorem]{Corollary}

\theoremstyle{definition}
\newtheorem{definition}[theorem]{Definition}

\newtheorem{example}[theorem]{Example}

\theoremstyle{remark}

\newtheorem{remark}[theorem]{Remark}
\newtheorem*{remark*}{Remark}

\numberwithin{equation}{section}

\usepackage[hypertexnames=false]{hyperref}

\title[Tautological rings and stabilisation]{Tautological rings and stabilisation}

\author{Oscar Randal-Williams}
\email{o.randal-williams@dpmms.cam.ac.uk}
\address{Centre for Mathematical Sciences\\
Wilberforce Road\\
Cambridge CB3 0WB\\
UK}
\begin{document}
\begin{abstract}
We construct a ring homomorphism comparing the tautological ring, fixing a point, of a closed smooth manifold with that of its stabilisation by $S^{2a} \times S^{2b}$.
\end{abstract}
\maketitle

\section{Introduction and statement of result}

\subsection{Tautological rings}

For a connected closed oriented smooth $d$-manifold $N$, the universal smooth fibre bundle with fibre $N$ may be described in terms of classifying spaces as
$$N \overset{i}\lra B\Diff^+(N, \star) \overset{\pi}\lra B\Diff^+(N).$$
Here $\Diff^+(N)$ denotes the topological group of orientation-preserving diffeomorphisms of $N$, and $\Diff^+(N, \star)$ denotes the subgroup of those diffeomorphisms fixing a marked point $\star \in N$. Assigning to a diffeomorphism fixing $\star \in N$ its derivative at this point gives a map
$$D_{\star} : B\Diff^+(N, \star) \lra B\GL_{d}^+(\bR) \simeq B\SO(d).$$
Using this we may pull back any cohomology class $c \in H^*(B\SO(d);\bQ)$ to give a class on $B\Diff^+(N, \star)$, which we continue to denote by $c$. We may then define classes
$$\kappa_c := \int_\pi c \in H^{|c|-d}(B\Diff^+(N);\bQ)$$
by integration along the fibres of the map $\pi$. These are known as tautological classes, $\kappa$-classes, or generalised Miller--Morita--Mumford classes. If $|c|=d$ then the degree zero cohomology class $\kappa_c$ is simply a characteristic number of $N$; the higher degree $\kappa_c$'s may be considered as analogues of characteristic numbers for families of manifolds.

The \emph{tautological ring} $R^*(N) \subset H^{*}(B\Diff^+(N);\bQ)$ is the subring generated by the classes $\kappa_c$. We may pull the classes $\kappa_c$ back along $\pi$ and hence also consider them as cohomology classes on $B\Diff^+(N, \star)$, where we continue to denote them by $\kappa_c$. A variant of $R^*(N)$, the \emph{tautological ring fixing a point} $R^*(N, \star) \subset H^{*}(B\Diff^+(N, \star);\bQ)$ is the subring generated by the classes $\kappa_c$ as well as the classes $c$.

\vspace{1ex}

\noindent\textbf{Context.} The rings $R^*(N)$ have been extensively studied in the case that $N$ is an oriented surface, as in this case $B\Diff^+(N)$ is a model for the moduli space of Riemann surfaces, cf.\ \cite{Mumford, Looijenga, Faber, Morita}. For manifolds of higher dimension they have recently been studied by Grigoriev, Galatius, and the author \cite{grigoriev-relations, galagriran-characteristic, RWPhenomena}, and in the case of 4-manifolds by Baraglia \cite{Baraglia}. In a different direction the \emph{vanishing} of tautological classes for various aspherical manifolds has been shown by Bustamante, Farrell, and Jiang \cite{BFJ}, and by Hebestreit, Land, L{\"u}ck, and the author \cite{HLLRW}. A variant of tautological rings for Poincar{\'e} complexes rather than manifolds has been studied by Prigge \cite{Prigge}.

\subsection{Main result}

The main result of this note concerns the case $d=2(a+b)$, and gives an explicit ring homomorphism $R^*(N \# S^{2a} \times S^{2b}, \star) \to R^*(N, \star)$. This is rather surprising because---as far as we can tell---there is no corresponding map $B\Diff^+(N, \star) \to B\Diff^+(N \# S^{2a} \times S^{2b}, \star)$, even at the level of rational cohomology groups.

In order to state our result we must first explain our conventions for describing the classes $c$. When $d=2n$ we have $H^*(B\SO(2n);\bQ) = \bQ[e, p_1, p_2, \ldots, p_{n-1}]$, the polynomial ring of the Euler class and Pontrjagin classes. There is a further Pontrjagin class, $p_n$, which agrees with $e^2$. Using this we may write any monomial in this ring as either $p_I$ or $e p_I$, with $I=(i_1, i_2, \ldots, i_r)$ having $1 \leq i_j \leq n$ and $p_I = p_{i_1} \cdots p_{i_r}$.

\begin{theorem}\label{thm:Stab}
Let $N$ be a $2(a+b)$-dimensional manifold. Then the formula
\begin{align*}
R^*(N \# S^{2a} \times S^{2b}, \star) &\lra R^*(N, \star)\\
\kappa_{p_I} & \longmapsto \kappa_{p_I}\\
\kappa_{ep_I} & \longmapsto \kappa_{ep_I} + 2 p_I\\
c &\longmapsto c
\end{align*}
gives a well-defined and surjective ring homomorphism.
\end{theorem}

\begin{remark}\ 
\begin{enumerate}[(i)]

\item The \emph{tautological ring fixing a disc} $R^*(N, D^d) \subset H^{*}(B\Diff^+(N, D^d);\bQ)$ is the subring generated by the classes $\kappa_c$. There are natural ring homomorphisms
$$R^*(N) \lra R^*(N,\star) \lra R^*(N, D^d)$$
sending $\kappa_c$ to $\kappa_c$, and the second map sending $c$ to 0. Considering $\Diff^+(N, D^d)$ as the group of diffeomorphisms of $N \setminus \mathrm{int}(D^d)$ fixing the boundary, there are natural maps $B\Diff^+(N, D^d) \to B\Diff^+(N \# M, D^d)$, and these induce ring homomorphisms
$$R^*(N \# M, D^d) \lra R^*(N, D^d),$$
sending $\kappa_c$ to $\kappa_c$. Theorem \ref{thm:Stab} may be viewed as a refinement of this map which does not require an entire disc to be fixed, but only a point.

\item Tautological rings can equally well be defined for homeomorphism groups of topological manifolds, but for our method smoothness is used in an essential way (we use that $\mathrm{Diff}^+(\bR^d)$ is homotopy equivalent to a compact Lie group).

\end{enumerate}
\end{remark}

Our method can more generally be used to compare tautological rings of $N$ and $N \# M$ when $M$ is a $2n$-manifold with an $n$-torus action satisfying certain cohomological hypotheses. In Section \ref{sec:General} we develop our construction in this generality, in Section \ref{sec:PfThm1} we verify the cohomological hypotheses in the case $M=S^{2a} \times S^{2b}$, thereby proving Theorem \ref{thm:Stab}, and in Section \ref{sec:CP2Ex} we explain the analogous result in the case $M=\mathbb{CP}^2$.

\begin{example}
Composing with the inclusion $i : N \to B\Diff^+(N,\star)$ of the fibre of the universal bundle gives a ring homomorphism
\begin{align*}
R^*(N \# S^{2k} \times S^{2k}, \star) &\lra H^*(N ; \bQ)\\
\kappa_{p_I} & \longmapsto 0\\
\kappa_{ep_I} & \longmapsto 2 p_I(TN)\\
c &\longmapsto c(TN).
\end{align*}
So if $p_I(TN) \neq 0 \in H^*(N;\bQ)$ then $\kappa_{ep_I} \neq 0 \in R^*(N \# S^{2k} \times S^{2k}, \star)$.
\end{example}

\begin{example}\label{ex:W}
Writing $W_g^{4k} = \#^g S^{2k} \times S^{2k}$, one consequence of Theorem \ref{thm:Stab} is that the sequence of Krull dimensions $\mathrm{Kdim} R^*(W_g^{4k}, \star)$ is non-decreasing with $g$. The morphism $R^*(W_g^{4k}) \to  R^*(W_g^{4k}, \star)$ is injective by the Becker--Gottlieb transfer \cite{beckegottl-transfer}: if $\pi : E \to B$ denotes the universal $W_g^{4k}$-bundle then the composition $\trf_\pi^* \circ \pi^*$ is multiplication by the Euler characteristic $\chi(W_g^{4k}) = 2+2g \neq 0$ and so $\pi^*$ is injective. By \cite[Theorem A (ii)]{RWPhenomena} the morphism $R^*(W_g^{4k}) \to  R^*(W_g^{4k}, \star)$ is in addition finite, so it follows that these rings have the same Krull dimensions and so the sequence $\mathrm{Kdim} R^*(W_g^{4k})$ is also non-decreasing with $g$. 

This is in distinction with the manifolds $W_g^{4k+2} = \#^g S^{2k+1} \times S^{2k+1}$, as in \cite{galagriran-characteristic} it was shown that the sequence $\mathrm{Kdim} R^*(W_g^{4k+2})$ is $2k+1, 0, 2k, 2k, 2k,\ldots$ for $g=0,1,2,3,4,\ldots$.
\end{example}

\begin{example}
We have $R^*(S^{4k})= \bQ[\kappa_{ep_1}, \kappa_{ep_2}, \ldots, \kappa_{ep_{2k}}]$ (see \cite[Section 5.3]{galagriran-characteristic}), so by \cite[Theorem A (ii)]{RWPhenomena} the ring $R^*(S^{4k}, \star)$ also has Krull dimension $2k$. Thus $\mathrm{Kdim} R^*(W_g^{4k}) \geq 2k$ for all $g \geq 0$.
\end{example}

\begin{example}
In \cite[Corollary 4.18]{RWPhenomena} it was shown that $\mathrm{Kdim} R^*(S^2 \times S^2)$ is either 3 or 4, so it follows that $\mathrm{Kdim} R^*(W_g^{4}) \geq 3$ for all $g \geq 1$.
\end{example}

\begin{example}
In \cite[Proposition 1]{GKT} it was shown that for the $K3$ manifold $K$ one has $\kappa_{\cL_{i+1}} \neq 0 \in R^*(K)$ for $1 \leq i \leq 8$. As $\chi(K)=24 \neq 0$ the map $R^*(K) \to R^*(K, \star)$ is injective by the Becker--Gottlieb transfer as in Example \ref{ex:W}, so  $\kappa_{\cL_{i+1}} \neq 0 \in R^*(K, \star)$ and hence $\kappa_{\cL_{i+1}} \neq 0 \in R^*(K \# g S^2 \times S^2)$ for all $g \geq 0$ and $1 \leq i \leq 8$. 
\end{example}

\subsection*{Acknowledgements}
The author was partially supported by the ERC under the European Union's Horizon 2020 research and innovation programme (grant agreement No.\ 756444), and by a Philip Leverhulme Prize from the Leverhulme Trust.

\section{The general construction}\label{sec:General}

\subsection{Parametrised connect-sum}

Let $(M, m_0)$ and $(N, n_0)$ be $d$-dimensional connected manifolds with marked points, and choose charts $\varphi_M : \bR^d \to M$ and  $\varphi_N : \bR^d \to N$ around these marked points. 

\begin{definition}
Let $\Diff^+(M, \varphi_M)$ denote the group of diffeomorphisms $f :M \to M$ for which there exists an $A \in \SO(d)$ such that $f \circ \varphi_M = \varphi_M \circ A$, equipped with the $C^\infty$-topology. Define $\Diff^+(N, \varphi_N)$ in the same way.

More generally for a subset $X \subset M \setminus \varphi_M(\bR^d)$ let $\Diff^+(M, \varphi_M, X)$ denote the subgroup which fixes $X$ pointwise.
\end{definition}

These are just slightly unusual models for the group of diffeomorphisms fixing a point, as follows.

\begin{lemma}
The inclusion $\Diff^+(M, \varphi_M) \to \Diff^+(M, m_0)$ is a weak homotopy equivalence. Similarly for $\Diff^+(N, \varphi_N)$ and $\Diff^+(M, \varphi_M, X)$.
\end{lemma}
\begin{proof}[Proof sketch]
Combine the facts (i) that the inclusion $\SO(d) \to \GL^+_d(\bR)$ into the space of invertible matrices with positive determinant is an equivalence, and (ii) that the space of Riemannian metrics on $M$ is contractible.
\end{proof}

Taking the derivative at the marked point $m_0 = \varphi_M(0)$ gives a homomorphism
$$D_{m_0} : \Diff^+(M, \varphi_M, X) \lra \SO(d),$$
and similarly with $n_0 = \varphi_N(0)$ gives
$$D_{n_0} : \Diff^+(N, \varphi_N) \lra \SO(d).$$
Let $r : \SO(d) \to \SO(d)$ be given by conjugating by a reflection: the induced map $Br : B\SO(d) \to B\SO(d)$ corresponds to reversing orientation. For compactness we write $\overline{D}_{m_0} := r \circ D_{m_0}$. Using these maps we can form a homotopy pullback square
\begin{equation}\label{eq:MainSq}
\begin{tikzcd}
\mathcal{G} \arrow[rr, "f_M"] \arrow[d, "f_N"] &&  B\Diff^+(M, \varphi_M,  X) \arrow[d, "\overline{D}_{m_0}"]\\
B\Diff^+(N, \varphi_N) \arrow[rr, "D_{n_0}"] && B\SO(d).
\end{tikzcd}
\end{equation}
The space $\mathcal{G}$ (for ``glue") is equipped with the following data: 
\begin{enumerate}[(i)]
\item an oriented orthogonal vector bundle $V \to \mathcal{G}$, 
\item a smooth oriented $N$-bundle $\pi_N : E_N \to \mathcal{G}$ with an orientation-preserving embedding ${s}_N : V \to E_N$,
\item a smooth oriented $M$-bundle $\pi_M : E_M \to \mathcal{G}$ with an orientation-reversing embedding ${s}_M : V \to E_M$ and a disjoint embedding $\mathcal{G} \times X \to E_M$.
\end{enumerate}
Furthermore $\mathcal{G}$ is the universal example of a space equipped with this data. For a characteristic class
$$c \in  H^*(B\SO(d);\bQ)$$
we write $c = c(V) \in H^*(\mathcal{G};\bQ)$ for its pullback to $\mathcal{G}$. We also write $q : S(\bR \oplus V) \to \mathcal{G}$ for the associated $d$-dimensional sphere bundle.

\begin{proposition}\label{prop:Cobordism}\mbox{}
\begin{enumerate}[(i)]
\item\label{it:Cobordism:1} There is a smooth oriented ${M} \# N$-bundle $\pi_{{M} \# N} : E_{{M} \# N} \to \mathcal{G}$ which is equipped with an embedding $\mathcal{G} \times X \to E_{{M} \# N}$ over $\mathcal{G}$. 

\item\label{it:Cobordism:2} There is a bundle of oriented cobordisms over $\mathcal{G}$
$$W : E_{{M}} \sqcup E_N \leadsto E_{{M} \# N} \sqcup S(\bR \oplus V),$$
 which is equipped with 
\begin{enumerate}[(a)]
\item an embedding $\mathcal{G} \times [0,1] \times X  \to W$ over $\mathcal{G}$, extending the embeddings of $\mathcal{G} \times X$ into $E_M$ and $E_{M \# N}$,
\item a $d$-dimensional oriented vector bundle which restricts to the vertical tangent bundle over the boundary.
\end{enumerate}
\end{enumerate}
\end{proposition}

We write 
$$f_{M\# N}: \mathcal{G} \lra B\Diff^+({M} \# N, X)$$
for the map which classifies the oriented $M \# N$-bundle given by Proposition \ref{prop:Cobordism} (\ref{it:Cobordism:1}).

\begin{proof}[Proof of Proposition \ref{prop:Cobordism}]
We form $E_{{M} \# N}$ by gluing together 
$$E_{{M}} \setminus \mathrm{int}(D(V)),\quad D^1 \times S(V),\quad \text{ and }\quad E_{N} \setminus \mathrm{int}(D(V))$$
along the natural identifications 
$$\partial (E_{{N}} \setminus \mathrm{int}(D(V))) \cong \{-1\} \times S(V)\quad \text{ and} \quad \partial (E_{{M}} \setminus \mathrm{int}(D(V))) \cong \{+1\} \times S(V),$$
 with smooth structure induced by the evident radial collar of $D(V) \subset V$. We have $\mathcal{G} \times X \subset E_{{M}} \setminus \mathrm{int}(D(V)) \subset E_{{M} \# N}$, and $E_{M \# N}$ admits a unique orientation compatible with those of $E_{{N}} \setminus \mathrm{int}(D(V))$ and $E_{{M}} \setminus \mathrm{int}(D(V))$. This establishes (\ref{it:Cobordism:1}).

To produce the cobordism data we first make a local construction. Let us write $U$ for the inner product space $\bR^d$ with an orientation $\det(U) \in \Lambda^d U$. Consider the elementary cobordism $W'_{\text{loc}}$ between $S^0 \times D(U)$ and $D^1 \times S(U)$,  realised $\SO(U)$-equivariantly by a codimension zero submanifold with corners
$$W'_{\text{loc}}  \subset \bR \times U$$
equipped with the Morse function $f(z, u) = |u|^2 - z^2$ and the orientation $\tfrac{\partial}{\partial z} \wedge \det(U)$. Remove the unit disc around $(0,0)$ from $W'_{\text{loc}}$ to obtain $W_{\text{loc}}$ as shown in Figure \ref{fig:Morse}, whose new boundary component is $S(U \oplus \bR)$. We consider $W_{\text{loc}}$ as a $\SO(U)$-equivariant oriented cobordism of manifolds with boundary
$$W_{\text{loc}} : S^0 \times D(U) \leadsto D^1 \times S(U) \sqcup S(\bR \oplus U).$$
As usual, we equip the boundary components of an oriented cobordism with the induced orientation on outgoing boundaries and the opposite of the induced orientation on incoming boundaries. (The induced orientation is that which agrees with prepending the outwards-pointing normal vector.) Thus in this case the orientation on $\{-1\} \times D(U)$ agrees with that inherited from $U$, that on $\{+1\} \times D(U)$ is the opposite of that inherited from $U$, and that on $S(\bR \oplus U)$ is the opposite of that induced from $D(\bR \oplus U)$.

\begin{figure}[ht]
\begin{center}
\includegraphics[width=10cm]{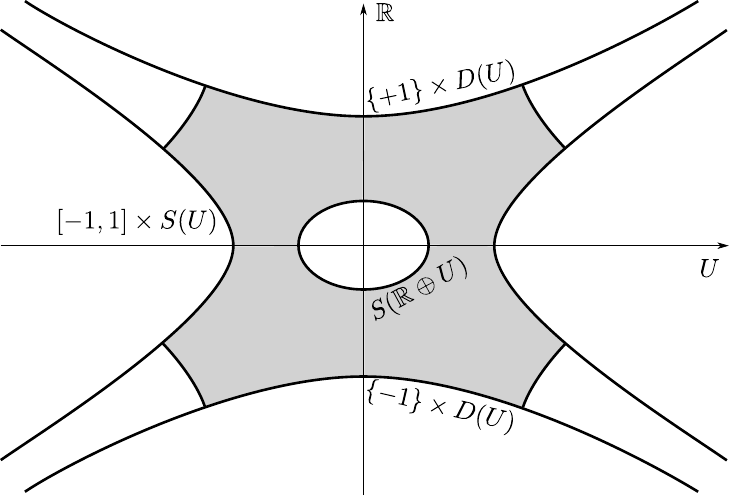}
\caption{The cobordism $W_{\text{loc}}$.}\label{fig:Morse}
\end{center}
\end{figure}

The kernel of the differential $Df$ defines a $\SO(U)$-equivariant $d$-dimensional oriented subbundle $\tau$ of $TW'_{\text{loc}}$ on the complement of the point $(0,0)$, restricting to the oriented tangent bundles of $\{-1\} \times D(U)$ and $D^1 \times S(U)$, and the oppositely-oriented tangent bundle over $\{+1\} \times D(U)$, respectively, as these are level sets of $f$. The bundle $\tau$ restricts to an oriented vector bundle of the same name on $W_{\text{loc}}$, and hence to an oriented vector bundle on $S(U \oplus \bR)$, which we now identify.

At the point $(z, u) \in S(\bR \oplus U)$ the differential $D_{(z,u)} f : \bR \oplus U \to \bR$ is given by inner product with $2(-z,u)$, so its kernel is identified with the tangent space $T_{(-z,u)} S(\bR \oplus U)$ though with the opposite orientation. As $(z,u) \mapsto (-z,u)$ gives a (orientation-reversing) diffeomorphism of $S(\bR \oplus U)$ commuting with the $\SO(U)$-action, there is $\SO(U)$-equivariant identification of oriented vector bundles between $\tau\vert_{S(\bR \oplus U)}$ and $T S(\bR \oplus U)$.

We implant this local construction as follows. Applying the local construction fibrewise the oriented orthogonal vector bundle $V \to \mathcal{G}$ gives a bundle of cobordisms of manifolds with boundary
$$W_{\text{loc}}(V) : S^0 \times D(V) \leadsto D^1 \times S(V) \sqcup S(\bR \oplus V)$$
over $\mathcal{G}$, equipped with a vector bundle $\tau(V)$ which agrees with the tangent bundle over the incoming and outgoing boundaries. We then construct $W$ as
$$\left([0,1] \times (E_{{M}} \setminus \mathrm{int}(D(V)) \sqcup E_{N} \setminus \mathrm{int}(D(V))) \right) \cup_{[0,1] \times S^0 \times S(V)} W_{\text{loc}}(V).$$
This is a bundle of oriented cobordisms $E_{{M}} \sqcup E_N \leadsto E_{{M} \# N} \sqcup S(\bR \oplus V)$, and contains $\mathcal{G} \times [0,1] \times X \subset [0,1] \times E_{{M}} \setminus \mathrm{int}(D(V))$. The bundle $\tau(V)$ extends to a vector bundle on $W$ by taking the vertical tangent bundle of $E_{{M}} \setminus \mathrm{int}(D(V)) \sqcup E_{N} \setminus \mathrm{int}(D(V))$. This establishes (\ref{it:Cobordism:2}).
\end{proof}

From now on let us suppose that $d=2n$, so that the available characteristic classes are
$$H^*(B\SO(2n);\bQ) = \bQ[p_1, p_2, \ldots, p_{n-1}, e].$$
As mentioned in the introduction, we use that $e^2 = p_n$ to write monomials in this ring as either $p_I$ or $e p_I$, with $I=(i_1, i_2, \ldots, i_r)$ having $1 \leq i_j \leq n$.

\begin{lemma}\label{lem:SphereCalc}
Let $V \to B$ be a $2n$-dimensional oriented orthogonal vector bundle, and $q : S(\bR \oplus V) \to B$ be the associated oriented $S^{2n}$-bundle. Then 
$$\kappa_{p_I}(q) = 0 \quad\quad \text{ and } \quad\quad \kappa_{ep_I}(q) = 2p_I(V).$$
\end{lemma}
\begin{proof}
We have $T_q S(\bR \oplus V) \oplus \bR \cong \bR \oplus q^*V$ and hence the cohomology classes $p_i(T_q S(\bR \oplus V)) = q^*p_i(V)$ are pulled back from the base. Then $\int_q p_I(T_q S(\bR \oplus V)) = \int_q q^*(p_I(V))=0$ by the projection formula (i.e.\ the fact that fibre integration is a map of modules over the cohomology of the base), and similarly
\begin{align*}
\int_q e(T_q S(\bR \oplus V))p_I(T_q S(\bR \oplus V)) &= \int_q e(T_q S(\bR \oplus V)) q^*(p_I(V))\\ 
&= \left(\int_q e(T_q S(\bR \oplus V))\right) p_I(V)\\
 &= \chi(S^{2n}) p_I(V) = 2 p_I(V).\qedhere
\end{align*}
\end{proof}

\begin{corollary}\label{cor:SumFormula}
There are identities
\begin{align*}
\kappa_{p_I}(\pi_{{M} \# N}) &= \kappa_{p_I}(\pi_M) + \kappa_{p_I}(\pi_N),\\
\kappa_{ep_I}(\pi_{{M} \# N}) &= \kappa_{ep_I}(\pi_M) + \kappa_{ep_I}(\pi_N) - 2p_I(V),
\end{align*}
in $H^*(\mathcal{G};\bQ)$.
\end{corollary}
\begin{proof}
Consider the bundle of cobordisms $W : E_{{M}} \sqcup E_N \leadsto E_{{M} \# N} \sqcup S(\bR \oplus V)$ constructed in Proposition \ref{prop:Cobordism} (\ref{it:Cobordism:2}), with its oriented $2n$-dimensional vector bundle which restricts to the vertical tangent bundles over the two ends. By Stokes' theorem, for any $c \in H^*(B\SO(2n);\bQ)$ we therefore have
$$\int_{\pi_{{M} \# N}} c(T_{\pi_{{M} \# N}} E_{{M} \# N}) + \int_q c(T_q S(\bR \oplus V))=  \int_{\pi_{{M}}} c(T_{\pi_{{M}}} E_{{M}}) + \int_{\pi_N} c(T_{\pi_N} E_N).$$
The result follows by using Lemma \ref{lem:SphereCalc}.
\end{proof}

\subsection{Torus actions}

If the $n$-torus $T$ acts on the $2n$-manifold $M$ fixing $m_0 \in M$ and $X = m_1 \in M$, then by choosing $\varphi_M$ to be an equivariant orthogonal chart around $m_0$ (obtained for example by exponentiating with respect to a $T$-invariant Riemannian metric) we have homomorphisms
$$T \overset{\phi}\lra \Diff^+(M, \varphi_M, m_1) \overset{\overline{D}_{m_0}}\lra \SO(2n).$$
We may then form the following commutative cube, in which the front face is \eqref{eq:MainSq}, the map $i^T : BT \to B\SO(2n)$ is $B(\overline{D}_{m_0} \circ \phi)$, and the remaining faces are developed by taking homotopy pullbacks.

\[
\begin{tikzcd}[row sep=2.5em, column sep = 0em]
\mathcal{G}_T \arrow[rr] \arrow[dr,swap, "i^\mathcal{G}"] \arrow[dd,swap] &&
  B\Diff^+_T(M, \varphi_M, m_1) \arrow[dd,swap, near start] \arrow[dr, "i^M"] \\
& \mathcal{G} \arrow[rr,crossing over, near start, "f_M"] &&
  B\Diff^+(M, \varphi_M, m_1) \arrow[dd, "\overline{D}_{m_0}"] \\
B\Diff^+_T(N, \varphi_N) \arrow[uu, dotted, bend left, "s^\phi"] \arrow[rr, near end, "D_{n_0}^T"] \arrow[dr,swap, "i^N"] && BT \arrow[dr,swap, "i^T"] \arrow[ur, "\phi", dotted] \arrow[uu, dotted, bend left] \\
& B\Diff^+(N, \varphi_N) \arrow[rr, "D_{n_0}"] \arrow[uu,<-,crossing over, near end, "f_N"]&& B\SO(2n)
\end{tikzcd}
\]

The lift $\phi$ determines a section $BT \to B\Diff^+_T(M, \varphi_M, m_1)$ and hence, by pullback, a section $s^\phi : B\Diff^+_T(N, \varphi_N) \to \mathcal{G}_T$. Furthermore, taking the derivative at $m_1$ gives a map
$$D_{m_1} : B\Diff^+(M, \varphi_M, m_1) \lra B\GL_{2n}^+(\bR) \simeq B\SO(2n).$$

\begin{lemma}\label{lem:CalcComp}
On tautological classes the composition
$$\varphi: B\Diff^+_T(N, \varphi_N) \overset{s^\phi}\lra \mathcal{G}_T \overset{i^\mathcal{G}}\lra \mathcal{G} \overset{f_{M \# N}}\lra B\Diff^+({M} \# N, m_1)$$
satisfies
\begin{align*}
\varphi^* \kappa_{p_I} &= (i^N)^*(\kappa_{p_I}) + (D_{n_0}^T)^*\phi^*(\kappa_{p_I})\\
\varphi^*\kappa_{ep_I} &= (i^N)^*(\kappa_{ep_I} - 2(D_{n_0})^*p_I) + (D_{n_0}^T)^*\phi^*(\kappa_{ep_I})\\
\varphi^* c &= (D_{n_0}^T)^*\phi^*(D_{m_1})^* c.
\end{align*}
\end{lemma}
\begin{proof}
By Corollary \ref{cor:SumFormula} we have
\begin{align*}
(f_{M \# N})^* \kappa_{p_I} &= \kappa_{p_I}(\pi_M) + \kappa_{p_I}(\pi_N)\\
(f_{M \# N})^*\kappa_{ep_I} &= \kappa_{ep_I}(\pi_M) + \kappa_{ep_I}(\pi_N) - 2p_I.
\end{align*}
When we pull this back to $B\Diff^+_T(N, \varphi_N)$ the classes $\kappa_c(\pi_N)$ can be written as $(i^N)^*\kappa_c$, and the classes $p_I$ can be written as $(i^N)^*(D_{n_0})^*p_I$. The classes $\kappa_c(\pi_M)$ pulled back to $B\Diff^+_T(N, \varphi_N)$ may be written as $(D^T_{n_0})^* \phi^* \kappa_c$.

Finally, Proposition \ref{prop:Cobordism} (\ref{it:Cobordism:2}) implies that
\[
\begin{tikzcd}
\mathcal{G} \arrow[r, "f_{M \# N}"] \arrow[d, "f_M"]& B\Diff^+({M} \# N, m_1) \arrow[d, "D_{m_1}"]\\
B\Diff^+(M, \varphi_M, m_1) \arrow[r, "D_{m_1}"] & B\SO(2n)
\end{tikzcd}
\]
commutes up to homotopy, which with the cube above shows that the composition
$$B\Diff^+_T(N, \varphi_N) \overset{s^\phi}\lra \mathcal{G}_T \overset{i^\mathcal{G}}\lra \mathcal{G} \overset{f_{M \# N}}\lra B\Diff^+({M} \# N, m_1) \overset{D_{m_1}}\lra B\GL_{2n}^+(\bR)$$
agrees with $B\Diff^+_T(N, \varphi_N) \overset{D^T_{n_0}}\to BT \overset{\phi}\to B\Diff^+(M, \varphi_M, m_1) \overset{D_{m_1}}\to B\GL_{2n}^+(\bR)$ up to homotopy.
\end{proof}

To proceed we require $\overline{D}_{m_0} \circ \phi : T \to\SO(2n)$ to be injective, in which case it is the inclusion of a maximal torus.

\begin{lemma}
The homomorphism $\overline{D}_{m_0} \circ \phi$ is injective if and only if $m_0 \in M$ is an isolated fixed point of this torus action.
\end{lemma}
\begin{proof}
If this homomorphism is not injective then its image is a torus of rank $\leq n-1$, which may therefore be conjugated into the maximal torus of $\SO(2n-1) \leq \SO(2n)$: in this case $T$ fixes a tangent vector at $m_0$, and since $T$ acts linearly in a chart around this point it follows that $m_0$ is not an isolated fixed point.

Conversely, if $m_0$ is not an isolated fixed point then this torus fixes a non-zero vector in $T_{m_0} M$, so lies in some $\SO(2n-1)$, and hence cannot be injective (by dimension of the maximal torus of $\SO(2n-1)$).
\end{proof}

Thus if $m_0$ is an isolated fixed point then the map $\overline{D}_{m_0} \circ \phi : T \to \SO(2n)$ is the inclusion of a maximal torus. We use this in the following way. The homotopy fibre of the map $i^N$ is then $\SO(2n)/T$ which has non-zero Euler characteristic (it is the order of the Weyl group of $\SO(2n)$), and therefore just as in Example \ref{ex:W} the Becker--Gottlieb 
transfer shows that the ring homomorphism
$$(i^N)^* : H^*(B\Diff^+(N, \varphi_N);\bQ) \lra H^*(B\Diff^+_T(N, \varphi_N);\bQ)$$
is injective. Thus:
\begin{enumerate}[(i)]
\item If for each $c \in H^*(B\SO(2n);\bQ)$ we have $(D_{n_0}^T)^*\phi^*(\kappa_{c}) = (i^N)^* q_c$ for some (unique) $q_c \in R^*(N, \star)$, then the function
\begin{align*}
R^*({M} \# N) &\lra R^*(N, \star)\\
\kappa_{p_I} &\longmapsto \kappa_{p_I} + q_{p_I}\\
\kappa_{ep_I} &\longmapsto \kappa_{ep_I} + q_{e p_I} - 2p_I 
\end{align*}
is a well-defined ring homomorphism.

\item If in addition for each $c \in H^*(B\SO(2n);\bQ)$ we have $(D_{n_0}^T)^*\phi^*(D_{m_1})^* c = (i^N)^* r_c$ for some (unique) $r_c \in R^*(N, \star)$, then the function
\begin{align*}
R^*({M} \# N, \star) &\lra R^*(N, \star)\\
\kappa_{p_I} &\longmapsto \kappa_{p_I} + q_{p_I}\\
\kappa_{ep_I} &\longmapsto \kappa_{ep_I} + q_{e p_I} - 2p_I\\
c &\longmapsto r_c
\end{align*}
is a well-defined ring homomorphism.
\end{enumerate}

For these to hold we must impose conditions on the torus action on $M$. We do not try to pursue this in its greatest generality, and instead treat two special cases.

\section{Proof of Theorem \ref{thm:Stab}: Stabilisation by $S^{2a} \times S^{2b}$}\label{sec:PfThm1}

We consider $S^{2k} = (\bR^{2k})^+$ with the usual $\SO(2k)$-action, and let the standard maximal torus $T = T^{a+b} = T^{a} \times T^{b} \leq \SO(2a) \times \SO(2b)$ act on $M = S^{2a} \times S^{2b}$. We write $\zeta$ for the corresponding oriented $2(a+b)$-dimensional representation of $T$, and $\bar{\zeta}$ for the same representation with opposite orientation.

\begin{lemma}\label{lem:FPData}
The $T$-action on $S^{2a} \times S^{2b}$ fixes $\{(0,0), (0,\infty), (\infty,0), (\infty,\infty)\}$. The $T$-representations at these points are all isomorphic to $\zeta$, but as oriented representations they are isomorphic to the $\zeta$ at $\{(0,0), (\infty,\infty)\}$ and to $\bar{\zeta}$ at $\{(0,\infty), (\infty,0)\}$.
\end{lemma}
\begin{proof}
The $T^k$-action on $S^{2k} = (\bR^{2k})^+$ fixes precisely 0 and $\infty$. The normal representation at 0 is the standard oriented representation $T^k \leq \SO(2k)$. The orientation-reversing reflection in the equator interchanges the fixed points $0$ and $\infty$ and commutes with the $T^k$-action, so the normal representation at $\infty$ is the opposite of the standard representation. Taking products gives the claimed description.
\end{proof}

\begin{proof}[Proof of Theorem \ref{thm:Stab}]
For $m_0 = (0,0) \in S^{2a} \times S^{2b}$ with orthogonal chart given by the product of the two open upper hemispheres, and $X = m_1 = (0, \infty) \in S^{2a} \times S^{2b}$, there is a corresponding map
$$\phi : BT \lra B\Diff^+(S^{2a} \times S^{2b}, \varphi_{S^{2a} \times S^{2b}},  m_1).$$

\begin{lemma}
This satisfies $\phi^*\kappa_{p_I}=0$ and $\phi^* \kappa_{ep_I} =  (i^T)^*(4 p_I)$.
\end{lemma}
\begin{proof}
We will use the localisation theorem in $T$-equivariant rational cohomology $H^*_T(-)$, in the following special case; see \cite[p.\ 366]{AP} for an exposition. 

Let $T$ act on an oriented $2n$-manifold $M$ with isolated fixed points $\{x_1, \ldots, x_r\}$, and let $e_j \in H^*_T = H^*_T(pt)$ denote the Euler class of the oriented $T$-representation on the tangent space $T_{x_j} M$. Let $S \subset H^*_T$ denote the multiplicative subset of non-zero elements. Then the $S$-localised fibre-integration map 
$$\int_\pi : S^{-1} H^*_T(M) \lra S^{-1}H^{*-2n}_T$$
for the bundle $\pi : M \times_T ET \to \{*\} \times_T ET = BT$ may be expressed as
$$\int_\pi t = \sum_{j=1}^r \frac{t\vert_{x_j}}{e_j}.$$
As $H^*_T$ is an integral domain this determines the unlocalised fibre-integration map.

We apply this to the $T$-action on $S^{2a} \times S^{2b}$. By Lemma \ref{lem:FPData}, tangent representations at all fixed points are isomorphic to $\zeta$, so have Pontrjagin classes $p_i(\zeta)$, but taking orientation into account the Euler classes of the tangent representations at $(0,0)$ and $(\infty, \infty)$ are $e(\zeta)$ and at $(0,\infty)$ and $(\infty, 0)$ are $e(\bar{\zeta}) = - e(\zeta)$. Thus we have
$$\phi^*\kappa_{p_I} = \int_\pi p_I(T(S^{2a} \times S^{2b})) = \frac{p_I(\zeta)}{e(\zeta)}+ \frac{p_I(\zeta)}{e(\zeta)} + \frac{p_I(\zeta)}{-e(\zeta)} + \frac{p_I(\zeta)}{-e(\zeta)}=0$$
and similarly
\begin{align*}
\phi^*\kappa_{ep_I}  &= \frac{e(\zeta)p_I(\zeta)}{e(\zeta)}+ \frac{e(\zeta)p_I(\zeta)}{e(\zeta)} + \frac{-e(\zeta)p_I(\zeta)}{-e(\zeta)} + \frac{-e(\zeta)p_I(\zeta)}{-e(\zeta)}\\
&= 4 p_I(\zeta) = (i_T)^*(4 p_I).\qedhere
\end{align*}
\end{proof}

As the oriented tangent representation at $m_1 \in S^{2a} \times S^{2b}$ is isomorphic but with opposite orientation to that at $m_0 \in S^{2a} \times S^{2b}$, so that $D_{m_1} \simeq \overline{D}_{m_0}$, we have
\begin{align*}
(D_{n_0}^T)^*\phi^*(\kappa_{p_I}) &= 0\\
(D_{n_0}^T)^*\phi^*(\kappa_{ep_I}) &= (i^N)^*(D_{n_0})^*(4 p_I)\\
(D_{n_0}^T)^*\phi^*(D_{m_1})^* ({c}) &= (i^N)^*(D_{n_0})^*(c).
\end{align*}
By the discussion above the formula
\begin{align*}
R^*({S^{2a} \times S^{2b}} \# N, \star) &\lra R^*(N, \star)\\
\kappa_{p_I} & \longmapsto \kappa_{p_I}\\
\kappa_{ep_I} & \longmapsto \kappa_{ep_I} + 2 p_I\\
c & \longmapsto c
\end{align*}
is then a well-defined ring homomorphism, which is clearly surjective. This proves Theorem \ref{thm:Stab}.
\end{proof}

\begin{remark}
In this argument we could have chosen the fixed point $m_1 = (\infty, \infty)$, whose tangential $T$-representation is oriented isomorphic with that at $m_0$. This has the disconcerting effect that the formula
\begin{align*}
R^*({S^{2a} \times S^{2b}} \# N, \star) &\lra R^*(N, \star)\\
\kappa_{p_I} & \longmapsto \kappa_{p_I}\\
\kappa_{ep_I} & \longmapsto \kappa_{ep_I} + 2 p_I\\
c &\longmapsto \bar{c}
\end{align*}
also gives a well-defined ring homomorphism, where $c \mapsto \bar{c}$ is the automorphism of $H^*(B\SO(2n);\bQ)$ induced by conjugation by a reflection.
\end{remark}

\section{Stabilisation by $\mathbb{CP}^2$}\label{sec:CP2Ex}

As a further example of the method, consider $M = {\mathbb{CP}}^2$ with the 2-torus action
\begin{align*}
S^1 \times S^1 \times {\mathbb{CP}}^2 &\lra {\mathbb{CP}}^2\\
(\xi_1, \xi_2, [z_0 : z_1: z_2]) &\longmapsto [z_0 : \xi_1 z_1 : \xi_2 z_2].
\end{align*}
We let $m_0 = [1 : 0 : 0]$ and $m_1 = [0 : 1 : 0]$; the third fixed point is $[0:0:1]$. The torus action gives a map
$$\phi : BT \lra B\Diff^+({\mathbb{CP}}^2, \varphi_{{\mathbb{CP}}^2}, m_1).$$
The map $D_{m_0} \circ \phi$ is induced by the standard inclusion $T \to \SO(4)$ of a maximal torus. If we let $H^*(BT;\bQ)=\bQ[x_1, x_2]$ then the representation at $m_0$ has $e = x_1 x_2$ and $p_1 = x_1^2 + x_2^2$. At the other fixed points we have $e = x_1(x_1-x_2)$ and $p_1 = x_1^2 + (x_2 - x_1)^2$, and $e = x_2(x_2-x_1)$ and $p_1 = x_2^2 + (x_1 - x_2)^2$, so by localisation in equivariant cohomology we may calculate
\begin{align*}
\phi^* \kappa_{e^a p_1^b} &=  \frac{(x_1 x_2)^a (x_1^2 + x_2^2)^b}{x_1 x_2} + \frac{(x_1^2 - x_1x_2)^a (x_1^2 + (x_2 - x_1)^2)^b}{x_1^2 - x_1x_2}\\
& \quad\quad\quad\quad\quad\quad + \frac{(x_2^2 - x_1 x_2)^a (x_2^2 + (x_1 - x_2)^2)^b}{x_2^2 - x_1 x_2}.
\end{align*}
When expanded out this is a polynomial in the $x_i$, and in fact is an even symmetric polynomial in these variables and so can be written in terms of $e = x_1 x_2$ and $p_1 = x_1^2 + x_2^2$. Call the resulting polynomial $q_{a,b}(e, p_1)$. The first few are
$$q_{0,1} = 3;\quad q_{1,0} = 3; $$
$$q_{0,2} = 7p_1 - 7e; \quad q_{1,1} = 4p_1 - 4e;\quad q_{2,0} = p_1-e;$$
$$q_{0,3} = 13(p_1^2 + e^2 -2 e p_1);\quad q_{1,2} = 6(p_1^2 + e^2 -2e p_1);\quad q_{2,1} = 2(p_1^2 + e^2 - 2 e p_1);$$
and $q_{3,0} = p_1^2 + e^2 - 2 e p_1$.

It then follows from our general discussion that the formula
\begin{align*}
R^*(N^4 \# \mathbb{CP}^2) &\lra R^*(N^4, \star)\\
\kappa_{e^a p_1^b} & \longmapsto \kappa_{e^a p_1^b} + \begin{cases}
 q_{a,b}(e, p_1) - 2 e^{a-1} p_1^b & \text{if $a$ is odd}\\
 q_{a,b}(e, p_1)  & \text{if $a$ is even}
\end{cases}
\end{align*}
gives a well-defined ring homomorphism. (This cannot be promoted to a ring homomorphism from $R^*(N^4 \# \mathbb{CP}^2, \star)$, because the Euler and Pontrjagin classes of the normal $T$-representation at $[0 : 1 : 0]$ (or $[0:0:1]$) cannot be expressed in terms of those at $m_0$.)

\begin{remark}
Baraglia \cite{Baraglia} has recently determined the tautological rings of $\mathbb{CP}^2$ and $\mathbb{CP}^2 \# \mathbb{CP}^2$, and more generally given a gauge-theoretic technique for obtaining relations in tautological rings of definite 4-manifolds. Up to a change of variables the polynomials $q_{a,b}$ also arise there.
\end{remark}

\bibliographystyle{amsalpha}
\bibliography{biblio}

\end{document}